\documentclass[12pt,reqno]{article}
\usepackage{authblk}
\usepackage{hyperref}
\usepackage{amssymb}
\usepackage{amsthm}
\usepackage{amsmath}
\usepackage{multirow}
\usepackage{float}
\usepackage{graphicx}
\usepackage{listings}
\usepackage{MnSymbol}
\usepackage{nicematrix}
\usepackage{titlesec}
\newtheorem{theorem}{Theorem}[section]

\newtheorem{proposition}[theorem]{Proposition}

\newtheorem{example}[theorem]{Example}

\newtheorem{remark}[theorem]{Remark}

\begin{document}
\title{\Large\bfseries Equivalence Problem for Non-Linearizable Third-Order ODEs with Four-Dimensional Lie Symmetry Subalgebras under Point Transformations}
\author[1]{Omar A. Abuloha\thanks{1227001@student.birzeit.edu}}
\author[1]{Marwan Aloqeili\thanks{maloqeili@birzeit.edu}}
\author[1]{Ahmad Y. Al-Dweik\thanks{aaldweik@birzeit.edu}}
\author[2]{F. M. Mahomed\thanks{Corresponding Author: Fazal.Mahomed@wits.ac.za}}
\affil[1]{Department of Mathematics, Birzeit University, Ramallah, Palestine}
\affil[2]{School of Computer Science and Applied Mathematics, 
University of the Witwatersrand, Johannesburg, Wits 2050,  South Africa}
\maketitle
\begin{abstract}
Cartan's equivalence method is applied to explicitly construct invariant coframes for \emph{four branches}, which are used to characterize all \emph{non-linearizable} third-order ODEs with a \emph{four-dimensional Lie  symmetry subalgebra} under point transformations. Additionally, we present a method for constructing the \emph{point  transformations} based on the derived invariant coframes. Examples are provided to illustrate our approach.
\end{abstract}
\bigskip
Keywords:  Cartan's equivalence method,  Equivalence problem,
Non-Linearizable third-order ODEs, Point symmetries, Point transformations.
\section{Introduction}
Lie’s classification \cite{Lie1888} shows that, over the complex domain, every scalar second-order ordinary differential equation (ODE) $u''=f(x,u,u')$ admitting a three-dimensional Lie algebra can be categorized into one of four distinct types.

Using the direct method, Ibragimov and Meleshko \cite{Ibragimov2007,Ibragimov2008} developed an invariant description of scalar second-order ordinary differential equations with three infinitesimal point symmetries, identifying candidate equations for each of the four types and studying them comprehensively. This study is important because any non-linearizable second-order ODE with a three-dimensional Lie symmetry algebra can be reduced, via point transformations, to one of these four canonical types. Recently, a new framework of the Cartan equivalence method has been employed by Al-Dweik et al. \cite{Al-Dweik2025} to explicitly construct three-dimensional invariant coframes for three branches, enabling the characterization of scalar second-order ODEs with a three-dimensional Lie algebra of point symmetries. Moreover, they introduced a procedure for constructing the corresponding point transformations from the obtained invariant coframes.

A natural extension of recent works is to investigate the equivalence of non-linearizable third-order ODEs possessing four-dimensional Lie symmetry subalgebras under both point and contact transformations.  These are two distinct problems; by contact we mean proper contact transformations, not point transformations. The objective of this paper is to provide an invariant characterization of non-linearizable third-order ODEs possessing four-dimensional Lie symmetry subalgebras under point transformations. 

A complete classification of scalar third-order ODEs which possess
non-similar real point symmetry Lie algebras $L_r$ of dimension $r$, where $r=1,\ldots,7$ can be obtained from many sources in the literature (see \cite{Mahomed1988,Gat1992,Schmucker1998,Ibragimov1996}). Classification of scalar third-order ODEs admitting non-similar real Lie algebras $L_r$ of dimension $r$, where $r=4, 5, 6, 7$ is given in the Appendix.

Based on Table 2 given in the Appendix, we conclude that any \emph{non-linearizable} third-order ODE $u''' = f(x,u,u',u'')$  under point transformations having a \emph{four}-dimensional Lie symmetry subalgebra  belongs to one of \emph{five} distinct types listed in the following table.  We denote $u', u''$ by $p, q$ respectively, in the following and thereafter.
\begin{table}[H]
\centering
\renewcommand{\arraystretch}{1.4}
\begin{tabular}{|l|l|l|}
\hline
\textbf{Algebra} & \textbf{Point symmetry realizations} & \textbf{Representative equations} \\  \hline\hline
$L_{4;2}$ & 
$X_1 = \partial_x,\quad X_2 = \partial_u,\quad X_3 = u\partial_u,$ & 
$u''' = \alpha \frac{q^2}{p},\quad \alpha \ne 0, \tfrac{3}{2}, 3$ \\
 & $X_4 = x\partial_x$ & \\ \hline
$L_{4;3}$ & 
\begin{tabular}[c]{@{}l@{}}$X_1 = \partial_u,\quad X_2 = \partial_x,\quad X_3 = x\partial_u,$\\ 
$X_4 = x\partial_x + (1 + b)u\partial_u,\quad b \ne 1$ \end{tabular} & 
$u''' =q^{\frac{b-2}{b-1}},\quad b \ne 1,2$ \\ \hline
$L_{4;5}^{II}$ & 
\begin{tabular}[c]{@{}l@{}}$X_1 = \partial_u,\quad X_2 = x\partial_x + u\partial_u,$\\ 
$X_3 = x\partial_x, \quad X_4 = 2xu\partial_x + u^2\partial_u$ \end{tabular} & 
$u''' = 3 \frac{q^2}{p} + \alpha \frac{(2xq + p)^{3/2}}{x^2 \sqrt{p}},\quad \alpha \ne 0$ \\ \hline

$L_{4;6}$ & 
\begin{tabular}[c]{@{}l@{}}$X_1 = \partial_u,\quad X_2 = x\partial_u,\quad X_3 = \partial_x,$\\ 
$X_4 = u\partial_x + \left(2u + \tfrac{1}{2}x^2\right)\partial_u$ \end{tabular} & 
$u''' =e^{-q}$\\ \hline

$L_{6;1}$ & 
\begin{tabular}[c]{@{}l@{}}$X_1 = \partial_x,\quad X_2 = x\partial_x,\quad X_3 = x^2\partial_x,$\\ 
$X_4 = \partial_u,\quad X_5 = u\partial_u,\quad X_6 = u^2\partial_u$ \end{tabular} & 
$u''' = \frac{3}{2} \frac{q^2}{p}$ \\ \hline

\end{tabular}
\caption{Canonical forms of non-linearizable third-order ODEs admitting four-dimensional Lie symmetry subalgebras under point transformation}
\end{table}
It should be noted here that the list presented in Table 1 will be shorter when the problem is considered under contact transformations, as some ODEs that are non-linearizable by point transformations may become linearizable under contact transformations.
\begin{remark}
In the complex plane $L_{6;1}$ is isomorphic and similar to  $L_{6;2}$. Indeed, the complex transformation $\bar x=u+ix, \bar u=x+iu$ maps the Equation $\bar{u}'''=\frac{3}{2} \frac{\bar{u}''^2}{\bar{u}'}$ to the Equation $u'''=\frac{3u'u''^2}{1+u'^2}$ .  So we exclude $L_{6;2}$ from Table 2 of the appendix.
\end{remark}
\begin{remark}
In the complex plane $L_{4;2}$ is isomorphic and similar to  $L_{4;4}$. Indeed, the complex transformation $\bar x=u+ix, \bar u=x+iu$ maps the Equation $\bar{u}'''=\alpha\frac{\bar{u}''^2}{\bar{u}'}$ to the Equation $u'''=\frac{3u'u''^2}{1+u'^2}+i(2\alpha-3)\frac{u''^2}{1+u'^2}$ .  So we exclude $L_{4;4}$ from Table 2 of the appendix.
\end{remark}
\begin{remark}
In Table 1, we omit all linear canonical forms listed in Table 2 of the appendix, as our focus is on non-linearizable third-order ODEs.
\end{remark}
\begin{remark}
The Lie algebras $L_{4;2}$ when $\alpha=3$ admits $7$ point symmetries and is therefore linearizable by a point transformation.  So we exclude it from Table 2 of the appendix.
\end{remark}
\begin{remark}
In Table 1, we  omit  $\alpha$ in the representative equation of $L_{4;3}$  listed in  Table 2 of the appendix, where it is noted that the transformation $\bar{x} =\frac{1}{\alpha} x,~\bar{u}=\frac{1}{\alpha^2}u$ maps the canonical from ${\bar{u}}'''= \alpha~ {\bar{u}}''^{\frac{b-2}{b-1}}$ to $u'''= {u''}^{\frac{b-2}{b-1}}$.
\end{remark}
\begin{remark}
In Table 1, we exclude  $\alpha$ in the representative equation of $L_{4;6}$  listed in  Table 2 of the appendix, where it is noted that the transformation $\bar{x} =\frac{1}{\alpha} x,~\bar{u}=\frac{1}{\alpha^2}u$ maps the canonical from ${\bar{u}}'''= \alpha~e^{-{\bar{u}}''}$ to $u'''= e^{-u''}$.
\end{remark}
This study is important because it complements recent research in the literature on the equivalence problem for linearizable third-order ODEs under both point and contact transformations  (see \cite{Chern1940, Neut2002, Ibragimov2005, Dweik2018_1, Dweik2019, Dweik2018_2}).

The paper is organized as follows. In the next section, we apply Cartan’s equivalence method to scalar third-order ODEs under the Lie group of \emph{point} transformations to explicitly construct  invariant coframes for \emph{four branches}, which are used to characterize all  third-order ODEs  given in Table 1.  In Section 3, main theorems are given. In Section 4,  a method for constructing the \emph{point} transformations is given based on the derived invariant coframes. Illustration of the theorems are amply explained through examples. A conclusion is presented in the last section.
\section{Utilization of Cartan's equivalence approach}
\setcounter{equation}{0}
The fundamental concepts, notations, and findings required in this section may be found in the foundational publications \cite{Olver1995,Neut2003}.
Let $\left(x, u, p=u^{\prime},q=u^{\prime\prime}\right) \in \mathbb{R}^4$ be local coordinates of the second-order jet space $\bf{J}^2$. Throughout this paper, the 1-forms $\pi^{\prime\kappa},~\kappa=1, \ldots, 8$  denote the modified Maurer-Cartan forms. Consider the following base coframe on $M=\bf{J}^2$
\begin{equation}\label{2.1}
\left(\begin{array}{l}
\omega^1\\\omega^2\\\omega^3\\\omega^4\\
\end{array}\right)
=\left(\begin{array}{c}
du-pdx\\
dp-qdx\\
dq-fdx\\
dx\\
\end{array}\right).
\end{equation}
The equivalence  of the third-order ODEs
\begin{equation}\label{2.2}
\begin{array}{l}
u^{\prime \prime \prime}=f\left(x, u, u^{\prime},u^{\prime \prime} \right), \quad \bar{u}^{\prime \prime\prime}=\bar{f}\left(\bar{x}, \bar{u}, \bar{u}^{\prime},\bar{u}^{\prime \prime}\right),
\end{array}
\end{equation}
under a point transformation
\begin{equation}\label{2.3}
\bar{x}=\varphi \left( x,u \right),~\bar{u} =\psi \left( x,u  \right),~~~\varphi_x\psi_u   -  \varphi_u \psi_x\neq0,\\
\end{equation}
can be translated, in local coordinates,  using the base coframe (\ref{2.1}) as the following equivalence conditions
\begin{equation}\label{2.4}
\Phi^*\left(\begin{array}{c}
\bar{\omega}^1 \\
\bar{\omega}^2 \\
\bar{\omega}^3\\
\bar{\omega}^4
\end{array}\right)=\left(\begin{array}{cccc}
a_1 & 0 & 0&0 \\
a_2 & a_3 & 0 &0\\
a_4 & a_5 & a_6&0\\
a_7&0&0&a_9
\end{array}\right)\left(\begin{array}{c}
\omega^1 \\
\omega^2 \\
\omega^3\\
\omega^4
\end{array}\right),
\end{equation}
for functions $a_i(x,u,p,q),~ i=1,2,3,4,5,6,7,9$,  where $\Phi^*$  denotes  the pullback defined by the second prolongation of the point transformation (\ref{2.3}). Therefore, the associated structure group is the eight-dimensional Lie group
\begin{equation}\label{2.5}
G=\left\{
\left(\begin{array}{cccc}
a_1 & 0 & 0&0 \\
a_2 & a_3 & 0 &0\\
a_4 & a_5 & a_6&0\\
a_7&0&0&a_9
\end{array}\right)  \Bigg \vert a_1a_3a_6a_9\neq 0
\right\}.
\end{equation}
We demonstrate that utilizing Cartan’s equivalence approach to examine the equivalence of scalar third-order ODEs, with each canonical representation listed in Table 1, results in four invariant coframes associated with four unique branches.
Let  $\theta$ denote the lifted coframe
\begin{equation}\label{2.6}
\begin{array}{ll}
\left(\begin{array}{c}
\theta^1 \\
\theta^2 \\
\theta^3\\
\theta^4
\end{array}\right)=\left(\begin{array}{cccc}
a_1 & 0 & 0&0 \\
a_2 & a_3 & 0 &0\\
a_4 & a_5 & a_6&0\\
a_7&0&0&a_9
\end{array}\right)\left(\begin{array}{c}
\omega^1 \\
\omega^2 \\
\omega^3\\
\omega^4
\end{array}\right).
\end{array}
\end{equation}
After absorption, the first structure equations for (\ref{2.6}) is
\begin{equation}\label{2.7}
\begin{array}{ll}
d\left(\begin{array}{c}
\theta^1 \\
\theta^2 \\
\theta^3\\
\theta^4
\end{array}\right)=\left(\begin{array}{cccc}
\pi^{\prime1} & 0 & 0&0 \\
\pi^{\prime2}&  \pi^{\prime3} & 0&0 \\
\pi^{\prime4}& \pi^{\prime5} & \pi^{\prime6}&0\\
\pi^{\prime7}&0&0&\pi^{\prime8}
\end{array}\right) \wedge\left(\begin{array}{c}
\theta^1 \\
\theta^2 \\
\theta^3\\
\theta^4
\end{array}\right)+\left(\begin{array}{c}
T_{24}^1 ~\theta^2 \wedge \theta^4 \\
T_{34}^2~\theta^3\wedge\theta^4 \\
0\\
0
\end{array}\right).
\end{array}
\end{equation}
The explicit formula for the essential torsion coefficients are $T_{24}^1=-\frac{a_1}{a_3 a_9}$ and $T_{34}^2=-\frac{a_3}{a_6 a_9}$
which can be normalized to -1 by setting $a_6=\frac{a^2_3} {a_1},a_9=\frac{a_1}{a_3}$. Thus, the structure group is reduced to
\begin{equation}\label{2.8}
{G}_1=\left\{
\left(\begin{array}{cccc}
a_1 & 0 & 0&0 \\
a_2 & a_3 & 0 &0\\
a_4 & a_5 & \frac{a^2_3} {a_1}&0\\
a_7&0&0&\frac{a_1}{a_3}
\end{array}\right)
 \Bigg \vert a_1a_3\neq 0
\right\},
\end{equation}
with the lifted one-forms
\begin{equation}\label{2.9}
\left(\begin{array}{l}
\theta^1 \\
\theta^2 \\
\theta^3 \\
\theta^4
\end{array}\right)=\left(\begin{array}{cccc}
a_1 & 0 & 0&0 \\
a_2 & a_3 & 0 &0\\
a_4 & a_5 & \frac{a^2_3} {a_1}&0\\
a_7&0&0&\frac{a_1}{a_3}
\end{array}\right)\left(\begin{array}{c}
\omega^1 \\
\omega^2 \\
\omega^3\\
\omega^4
\end{array}\right) .
\end{equation}
The  structure equations become  after a \emph{second loop} of  reduction and  absorption,
\begin{equation}\label{2.10}
d\left(\begin{array}{c}
\theta^1 \\
\theta^2 \\
\theta^3\\
\theta^4
\end{array}\right)=\left(\begin{array}{cccc}
 \pi^{\prime1} & 0 & 0&0 \\
\pi^{\prime2} &  \pi^{\prime3} & 0&0 \\
\pi^{\prime4} &\pi^{\prime5}& 2\pi^{\prime3}-\pi^{\prime1}&0\\
\pi^{\prime6}&0&0&\pi^{\prime1}-\pi^{\prime3}
\end{array}\right) \wedge\left(\begin{array}{c}
\theta^1 \\
\theta^2 \\
\theta^3\\
\theta^4
\end{array}\right)+\left(\begin{array}{c}
- \theta^2 \wedge \theta^4 \\
 - \theta^3 \wedge \theta^4 \\
T_{34}^3~\theta^3 \wedge \theta^4 \\
0
\end{array}\right).
\end{equation}
where  $T_{34}^3=\frac{a_3^2I_1-3a_1a_5+3a_2a_3}{a_1a_3}$ is the only essential torsion coefficient that appears in the structure equations (\ref{2.10}), which can be normalized to 0  by setting $a_5=\frac{a_2a_3}{a_1}+\frac{a_3^2}{3a_1}I_1$, where $I_1=-f_q$. 
Therefore, the structure group can be reduced to
\begin{equation}\label{2.11}
{G}_2=\left\{
\left(\begin{array}{cccc}
a_1 & 0 & 0&0 \\
a_2 & a_3 & 0 &0\\
a_4 &\frac{a_2a_3}{a_1}+\frac{a_3^2}{3a_1}I_1& \frac{a^2_3} {a_1}&0\\
a_7&0&0&\frac{a_1}{a_3}
\end{array}\right)
 \Bigg \vert a_1a_3\neq 0
\right\},
\end{equation}
with the lifted one-forms
\begin{equation}\label{2.12}
\left(\begin{array}{l}
\theta^1 \\
\theta^2 \\
\theta^3 \\
\theta^4
\end{array}\right)=\left(\begin{array}{cccc}
a_1 & 0 & 0&0 \\
a_2 & a_3 & 0 &0\\
a_4 &\frac{a_2a_3}{a_1}+\frac{a_3^2}{3a_1}I_1& \frac{a^2_3} {a_1}&0\\
a_7&0&0&\frac{a_1}{a_3}
\end{array}\right)\left(\begin{array}{c}
\omega^1 \\
\omega^2 \\
\omega^3\\
\omega^4
\end{array}\right) .
\end{equation}
Proceeding to a \emph{third loop} of reduction, the structure equations become after absorption,
\begin{equation}\label{2.13}
\begin{array}{ll}
d\left(\begin{array}{c}
\theta^1 \\
\theta^2 \\
\theta^3\\
\theta^4
\end{array}\right)=\left(\begin{array}{cccc}
\pi^{\prime1} & 0 & 0&0 \\
\pi^{\prime2}&  \pi^{\prime3} & 0&0 \\
 \pi^{\prime4}& \pi^{\prime2} & 2\pi^{\prime3}-\pi^{\prime1}&0\\
 \pi^{\prime5}&0&0&\pi^{\prime1}-\pi^{\prime3}
\end{array}\right) \wedge\left(\begin{array}{c}
\theta^1 \\
\theta^2 \\
\theta^3\\
\theta^4
\end{array}\right)+\left(\begin{array}{c}
-\theta^2 \wedge \theta^4 \\
-\theta^3\wedge\theta^4 \\
T_{24}^3~\theta^2 \wedge \theta^4\\
T_{24}^4~\theta^2 \wedge \theta^4
\end{array}\right)\end{array}.
\end{equation}
The explicit formula for the essential torsion coefficients are $T_{24}^3=\frac{a_3^2I_2-2a_4a_1+a_2^2}{a_1^2}$, $T^4_{24}=\frac{a_1I_3-a_7a_3}{a_1a_3}$. We get the normalizations $T^3_{24}=0$ and $T^4_{24}=0$ by setting  $a_4=\frac{a_2^2}{2a_1}+\frac{a_3^2}{2a_1}I_2$ and  $a_7=\frac{a_1}{a_3}I_3$, respectively, where  $I_2=-\frac{2}{9}I_1^2-f_p-\frac{1}{3}\hat{D}_xI_1$ and $I_3=-\frac{1}{6}I_{1_q}$. 
Hence, the structure group is reduced to
\begin{equation}\label{2.14}
{G}_3=\left\{
\left(\begin{array}{cccc}
a_1 & 0 & 0&0 \\
a_2 & a_3 & 0 &0\\
\frac{a_2^2}{2a_1}+\frac{a_3^2}{2a_1}I_2&\frac{a_2a_3}{a_1}+\frac{a_3^2}{3a_1}I_1& \frac{a^2_3} {a_1}&0\\
\frac{a_1}{a_3} I_3&0&0&\frac{a_1}{a_3}
\end{array}\right)
 \Bigg \vert a_1a_3\neq 0
\right\},
\end{equation}
with the lifted one-forms
\begin{equation}\label{2.15}
\left(\begin{array}{l}
\theta^1 \\
\theta^2 \\
\theta^3 \\
\theta^4
\end{array}\right)=\left(\begin{array}{cccc}
a_1 & 0 & 0&0 \\
a_2 & a_3 & 0 &0\\
\frac{a_2^2}{2a_1}+\frac{a_3^2}{2a_1}I_2&\frac{a_2a_3}{a_1}+\frac{a_3^2}{3a_1}I_1& \frac{a^2_3} {a_1}&0\\
\frac{a_1}{a_3}I_3&0&0&\frac{a_1}{a_3}
\end{array}\right)\left(\begin{array}{c}
\omega^1 \\
\omega^2 \\
\omega^3\\
\omega^4
\end{array}\right) .
\end{equation}
The structure equations become after a \emph{fourth loop} of reduction and  absorption,
\begin{equation}\label{2.16}
\begin{array}{ll}
d\left(\begin{array}{c}
\theta^1 \\
\theta^2 \\
\theta^3\\
\theta^4
\end{array}\right)=\left(\begin{array}{cccc}
\pi^{\prime1} & 0 & 0&0 \\
\pi^{\prime2}&  \pi^{\prime3} & 0&0 \\
0& \pi^{\prime2} & 2\pi^{\prime3}-\pi^{\prime1}&0\\
0&0&0&\pi^{\prime1}-\pi^{\prime3}
\end{array}\right) \wedge\left(\begin{array}{c}
\theta^1 \\
\theta^2 \\
\theta^3\\
\theta^4
\end{array}\right)+\left(\begin{array}{c}
-\theta^2 \wedge \theta^4 \\
-\theta^3\wedge\theta^4 \\
T_{14}^3~\theta^1 \wedge \theta^4\\
T_{12}^4~\theta^1 \wedge \theta^2+T^4_{13}~\theta^1 \wedge \theta^3
\end{array}\right).
\end{array}
\end{equation}
The essential  torsion coefficients are 
\begin{equation}\label{2.17}
\begin{array}{ccc}
T^3_{14}=\frac{a_3^3}{a_1^3}I_4,
&T^4_{12}=\frac{1}{a_3^2}I_5-\frac{a_2}{a_3^3}I_6,
&T^4_{13}=\frac{a_1}{a_3^3}I_6,\\
\end{array}
\end{equation}
where 
\begin{equation}\label{2.18}
\begin{array}{ccc}
I_4=-\frac{1}{3}I_1I_2-f_u-\frac{1}{2}\hat{D}_xI_2,
&I_5=\frac{1}{3}I_1I_{3_q}-I_3^2-I_{3_p},
&I_6=-I_{3_q}.
\end{array}
\end{equation}
Any multiplicative factor of the absolute invariant that is independent of the group parameters constitutes a relative invariant. A relative invariant refers to an expression that, when subjected to a point transformation, is transformed into a multiple of itself. If the expression equals zero prior to the transformation, it will continue to equal zero subsequently. As a result, we derive the subsequent four branches.
\subsection{Branch $I_{4}\ne 0$ and $ I_{6}\neq 0$:}
 We can eliminate the group parameters $a_{1} , a_{2}, a_{3}$ by normalizing 
$T^3_{14}=\frac{1}{r}, T^4_{12}=0$  and $T^4_{13}=\frac{1}{s}$ by setting 
\begin{equation}\label{2.19}
\begin{array}{lll}
a_{1}=J_{4}^3J_{6},& a_{2}=s \frac{J_4I_5}{J_6},& a_{3}=J_4J_6,
\end{array}
\end{equation}
where $r I_4=J_4^6$, $s I_6=J_6^2$ and $r, s$ are any nonzero numbers, this
gives invariant coframe on the space  $M=\bf{J}^2.$
\begin{equation}\label{2.20}
\begin{pmatrix}
\theta^{1}\\
\theta^{2}\\
\theta^{3}\\
\theta^4\\
\end{pmatrix}=
\left(\begin{array}{cccc}
J_{4}^3J_{6} & 0 & 0&0 \\
\frac{sJ_4I_5}{J_6} & J_4J_6 & 0 &0 \\
\frac{I_2J_6^4+I_5^2}{2J_4J_6^3}&\frac{I_1J_6^2+3s I_5}{3J_4J_6} &\frac{J_6}{J_4}&0\\
I_3J_4^2&0&0&J_4^2
\end{array}\right)
\begin{pmatrix}
\omega^{1}\\
\omega^{2}\\
\omega^{3}\\
\omega^{4}
\end{pmatrix}.
\end{equation}
The invariant coframe corresponding to the canonical forms $u'''= e^{-u''}$ and $u'''= u''^{\frac{b-2}{b-1}}$, $b\ne-1,0,\frac{1}{2},1,2$   in this branch satisfies the structure equations with constant structure functions. 
\subsubsection{The canonical form $u'''= e^{-u''}.$}
By setting $r=4$ and $s=1$, we determine
\begin{equation}\label{2.21}
\begin{array}{ll}
&d\theta^{1}=\frac{1}{2}A^5B~\theta^{1}\wedge\theta^{2}+4AB~\theta^1\wedge\theta^3+\frac{9}{4}A^4~\theta^1\wedge\theta^4-\theta^2\wedge\theta^4,\\
&d\theta^{2}=\frac{17}{16}A^3B~\theta^{1}\wedge\theta^{2}-2A^5B~\theta^{1}\wedge\theta^{3}-\frac{11}{32}A^2~\theta^{1}\wedge\theta^{4}+2AB~\theta^2\wedge\theta^3+2A^4~\theta^{2}\wedge\theta^{4}\\
&~~~~~~~-\theta^{3}\wedge\theta^{4},\\
&d\theta^{3}=\frac{1}{32}AB~\theta^{1}\wedge\theta^{2}+\frac{17}{8}A^3B~\theta^1\wedge\theta^{3}+\frac{1}{4}\theta^{1}\wedge\theta^{4}-\frac{1}{2}A^5B~\theta^2\wedge\theta^3-\frac{11}{32}A^2~\theta^2\wedge\theta^4\\
&~~~~~~+\frac{7}{4}A^4~\theta^3\wedge\theta^4,\\
&d\theta^{4}=\theta^1\wedge\theta^3+\frac{25}{16}A^3B~\theta^1\wedge\theta^4-A^5B~\theta^2\wedge\theta^4-2AB~\theta^3\wedge\theta^4,
\end{array}
\end{equation}
where $A^6=1, B^2=1$.
\subsubsection{The canonical form $u'''= u''^{\frac{b-2}{b-1}}, b\ne-1,0,\frac{1}{2},1,2.$}
By inserting $r=-8$ and $s=-1$, we deduce
\begin{equation}\label{2.22}
\begin{array}{ll}
&d\theta^{1}=-\frac{\alpha }{2\beta b^2}A^5 B~\theta^{1}\wedge\theta^{2}+\frac{2\beta(b+1)}{\alpha}AB~\theta^1\wedge\theta^3+\frac{8b^2+3b-2}{4\beta^2b}A^4~\theta^1\wedge\theta^4-\theta^2\wedge\theta^4,\\
&d\theta^{2}=-\frac{4b^5+4b^4-11b^3-34b^2+12b+8}{16\alpha\beta^3b^2}A^3B~\theta^{1}\wedge\theta^{2}+\frac{\alpha(3b+1)}{2\beta b^2}A^5B~\theta^{1}\wedge\theta^{3}\\
&+\frac{4b^4+4b^3-27b^2+4b+4}{32\beta^4 b^2}A^2~\theta^{1}\wedge\theta^{4}+\frac{2\beta b}{\alpha}AB~\theta^2\wedge\theta^3+\frac{1+b}{\beta^2}A^4~\theta^{2}\wedge\theta^{4}-\theta^{3}\wedge\theta^{4},\\
&d\theta^{3}=-\frac{\alpha^3(4b^4+8b^3-3b^2-9b-2)}{64\beta^5b^5}AB~\theta^{1}\wedge\theta^{2}-\frac{4b^5-5b^3-13b^2-16b-4}{16\alpha\beta^3b^2}A^3B~\theta^1\wedge\theta^{3}\\
&-\frac{1}{8}\theta^{1}\wedge\theta^{4}+\frac{\alpha}{2\beta b }A^5B~\theta^2\wedge\theta^3+\frac{4b^4+4b^3-27b^2+4b+4}{32\beta^4b^2}A^2~\theta^2\wedge\theta^4+\frac{5b+2}{4\beta^2b}A^4~\theta^3\wedge\theta^4,\\
&d\theta^{4}=-\theta^1\wedge\theta^3-\frac{20b^4-30b^3-23b^2+4b+4}{16\alpha\beta^3b^2}A^3B~\theta^1\wedge\theta^4+\frac{\alpha(b+1)}{2\beta b^2}A^5B~\theta^2\wedge\theta^4\\
&-\frac{2\beta}{\alpha}AB~\theta^3\wedge\theta^4,
\end{array}
\end{equation}
where $A^6=1, B^2=1$ and $\alpha^2=b^2-2b, \beta^6=2b^3-3b^2-3b+2$.
It should be stressed here that  $T^1_{12}T^2_{23}=-\frac{1}{b}$ is an invariant relation, which enables the determination of the parameter $b$ appearing in the canonical form.
\begin{remark}
The cases $b=-1,0,\frac{1}{2}$ have been excluded from the representative equation of $L_{4;3}$,  as they do not belong to this branch. These cases occur in other branches and are addressed later.
\end{remark}
This proves Theorem  \ref{Th3.1} in section 3, see \cite{Olver1995} [Theorem 8.15, page 268].
\subsection{Branch $I_{6}= 0$ and $I_4\ne0, I_{5}\neq 0$:}
 We can eliminate the group parameters $a_{1} , a_{3}$ using the normalizations 
$T^3_{14}=\frac{1}{r}, T^4_{12}=-1$  by setting 
\begin{equation}\label{2.23}
\begin{array}{lll}
a_{1}=J_{4}J_5,& a_{3}=J_5,
\end{array}
\end{equation}
where $r I_4=J_4^3, I_5=-J_5^2$ and $r$ is any nonzero number. After performing another loop of reduction and absorption, we obtain the structure equations 
\begin{equation}\label{2.24}
\begin{array}{ll}
d\left(\begin{array}{c}
\theta^1 \\
\theta^2 \\
\theta^3\\
\theta^4
\end{array}\right)=\left(\begin{array}{cccc}
0 & 0 & 0&0 \\
\pi^{\prime1}& 0& 0&0 \\
0& \pi^{\prime1} &0&0\\
0&0&0&0
\end{array}\right) \wedge\left(\begin{array}{c}
\theta^1 \\
\theta^2 \\
\theta^3\\
\theta^4
\end{array}\right)\\
+\left(\begin{array}{c}
T^1_{12}~\theta^1 \wedge \theta^2+T^1_{13}~\theta^1 \wedge \theta^3+T^1_{14}~\theta^1 \wedge \theta^4-\theta^2 \wedge \theta^4 \\
T^2_{24}~\theta^2 \wedge \theta^4-\theta^3\wedge\theta^4 \\
T_{13}^3~\theta^1 \wedge \theta^3+\theta^1 \wedge \theta^4+T^3_{23}~\theta^2 \wedge \theta^3+T^3_{34}~\theta^3 \wedge \theta^4\\
-\theta^1 \wedge \theta^2+T^4_{14}~\theta^1 \wedge \theta^4+T^4_{24}~\theta^2 \wedge \theta^4+T^4_{34}~\theta^3 \wedge \theta^4
\end{array}\right),
\end{array}
\end{equation}
and the essential torsion coefficients
\begin{equation}\label{2.25}
\begin{array}{ll}
T^1_{13}=-\frac{J_{4_q}}{J_5}, &T^{4}_{24}=I_7-\frac{a_2J_{4_q}}{J_4J_5^2},
\end{array}
\end{equation}
where
\begin{equation}\label{2.26}
\begin{array}{ll}
I_7=-\frac{1}{3J_4J_5^2}(3I_3J_4J_5+I_1J_{4_q}J_5-3J_{4_p}J_5).\\
\end{array}
\end{equation}
It is noted that $J_{4_q}$ is a relative invariant and $L_{4;2}, L_{4;3}$ when $b=0$ belong to the sub-branch $J_{4_q}\ne0$. So, normalizing the essential torsion coefficient $T^{4}_{24}=0$ by setting 
\begin{equation}\label{2.27}
\begin{array}{ll}
a_{2}=\frac{J_4J_5^2I_7}{J_{4_q}},
\end{array}
\end{equation}
gives the  invariant coframe on the space  $M=\bf{J}^2$
\begin{equation}\label{2.28}
\begin{pmatrix}
\theta^{1}\\
\theta^{2}\\
\theta^{3}\\
\theta^4\\
\end{pmatrix}=
\left(\begin{array}{cccc}
J_4J_5 & 0 & 0&0 \\
\frac{J_4J_5^2I_7}{J_{4_q}} &J_5& 0 &0 \\
\frac{J_5(J_4^2J_5^2I_7^2+I_2J_{4_q}^2)}{2J_{4_q}^2J_4}&\frac{J_5(3J_4J_5I_7+I_1J_{4_q})}{3J_{4_q}J_4} &\frac{J_5}{J_4}&0\\
I_3J_4&0&0&J_4
\end{array}\right)
\begin{pmatrix}
\omega^{1}\\
\omega^{2}\\
\omega^{3}\\
\omega^{4}
\end{pmatrix}.
\end{equation}
The invariant coframe corresponding to the following two canonical forms in this branch satisfies structure equations with constant structure functions.
\subsubsection{The canonical form $u'''=\alpha\frac{u''^2}{u'},\alpha\ne0,\frac{3}{2},3.$}
By setting $r=1$, we deduce
\begin{equation}\label{2.29}
\begin{array}{ll}
d\theta^{1}&=-\frac{2\alpha-3}{\gamma} B~\theta^{1}\wedge\theta^{2}-\frac{\beta}{\gamma} AB~\theta^1\wedge\theta^3-\frac{2\alpha-3}{\beta}A^2~\theta^1\wedge\theta^4-\theta^2\wedge\theta^4,\\
d\theta^{2}&=-\frac{3(\gamma^2+3)}{\beta\gamma}A^2B~\theta^{1}\wedge\theta^{2}-\frac{2\alpha-3}{\gamma} B~\theta^{1}\wedge\theta^{3}-\frac{3(\gamma^2+3)}{2\beta^2}A~\theta^{1}\wedge\theta^{4}-\frac{\beta^2}{\gamma^2}A^2~\theta^{2}\wedge\theta^{4}\\
&~~-\theta^3\wedge\theta^4,\\
d\theta^{3}&=-\frac{(2\alpha-3)(\gamma^2+9)}{2\beta^2\gamma}AB~\theta^{1}\wedge\theta^{2}-\frac{3(\gamma^2+3)}{2\beta\gamma}A^2B~\theta^1\wedge\theta^{3}+\theta^{1}\wedge\theta^{4}-\frac{3(\gamma^2+3)}{2\beta^2}A~\theta^2\wedge\theta^4\\
&~~-\frac{2\alpha-3}{\beta}A^2 ~\theta^3\wedge\theta^4,\\
d\theta^{4}&=-\theta^1\wedge\theta^2-\frac{5\gamma^2+9}{2\beta\gamma}A^2B~\theta^1\wedge\theta^4+\frac{\beta}{\gamma} AB~\theta^3\wedge\theta^4,
\end{array}
\end{equation}
where   $A^3=1,B^2=1$ and $\beta^3=2\alpha^3-9\alpha^2+9\alpha,~\gamma^2=\alpha^2-3\alpha.$\\
\\
It should be noted here that  $\frac{J_{4_q}}{J_5}=\frac{\beta}{\gamma}$ is an invariant relation, which enables the determination of the parameter $\alpha$ appearing in the canonical form.
\subsubsection{The canonical form $u'''=u''^2.$}
By letting $r=4$, we obtain
\begin{equation}\label{2.31}
\begin{array}{ll}
&d\theta^{1}=-2B~\theta^{1}\wedge\theta^{2}-2AB~\theta^1\wedge\theta^3-A^2~\theta^1\wedge\theta^4-\theta^2\wedge\theta^4,\\
&d\theta^{2}=-\frac{3}{2}A^2B~\theta^{1}\wedge\theta^{2}-2B~\theta^{1}\wedge\theta^{3}-\frac{3}{8}A~\theta^{1}\wedge\theta^{4}-A^2~\theta^2\wedge\theta^4-\theta^3\wedge\theta^4,\\
&d\theta^{3}=-\frac{1}{4}AB~\theta^{1}\wedge\theta^{2}-\frac{3}{4}A^2B~\theta^1\wedge\theta^{3}+\frac{1}{4}~\theta^{1}\wedge\theta^{4}-\frac{3}{8}A~\theta^2\wedge\theta^4-A^2~\theta^3\wedge\theta^4,\\
&d\theta^{4}=-\theta^1\wedge\theta^2-\frac{5}{4}A^2B~\theta^1\wedge\theta^4+2AB~\theta^3\wedge\theta^4,
\end{array}
\end{equation}
where $A^3=1, B^2=1$.\\
\\
This proves Theorem  \ref{Th3.2} in section 3, see \cite{Olver1995} [Theorem 8.15, page 268].
\subsection{Branch $I_{4}= 0$ and $I_6\ne0$:}
 We can eliminate the group parameters $a_{1} , a_{2}$ by normalizing the essential torsion coefficients as 
$T^4_{12}=0, T^4_{13}=1$  by  setting
\begin{equation}\label{2.32}
\begin{array}{lll}
a_{1}=\frac{a_3^3}{I_6},\quad a_{2}=\frac{a_3I_5}{I_6}.
\end{array}
\end{equation}
After performing another loop of reduction and absorption, we get the structure equations 
\begin{equation}\label{2.33}
\begin{array}{ll}
d\left(\begin{array}{c}
\theta^1 \\
\theta^2 \\
\theta^3\\
\theta^4
\end{array}\right)=\left(\begin{array}{cccc}
3\pi^{\prime1} & 0 & 0&0 \\
0& \pi^{\prime1}& 0&0 \\
0& 0 &-\pi^{\prime1}&0\\
0&0&0&2\pi^{\prime1}
\end{array}\right) \wedge\left(\begin{array}{c}
\theta^1 \\
\theta^2 \\
\theta^3\\
\theta^4
\end{array}\right)\\
+\left(\begin{array}{c}
-\theta^2 \wedge \theta^4 \\
T^2_{13}~\theta^1 \wedge \theta^3+T^2_{14}~\theta^1 \wedge \theta^4+T^2_{23}~\theta^2 \wedge \theta^3+T^2_{24}~\theta^2 \wedge \theta^4-\theta^3\wedge\theta^4 \\
T_{12}^3~\theta^1 \wedge \theta^2+T^3_{13}~\theta^1 \wedge \theta^3+T^3_{23}~\theta^2 \wedge \theta^3+T^3_{24}~\theta^2 \wedge \theta^4+T^3_{34}~\theta^3 \wedge \theta^4\\
\theta^1 \wedge \theta^3+T^4_{14}~\theta^1 \wedge \theta^4+T^4_{24}~\theta^2 \wedge \theta^4+T^4_{34}~\theta^3 \wedge \theta^4
\end{array}\right),
\end{array}
\end{equation}
 and the essential torsion coefficient $T^{2}_{13}=\frac{I_7}{a_3}$, where
$I_7=-I_3-\left(\frac{I_5}{I_6}\right)_q$ is a relative invariant.
We note that $L_{4;5},L_{4;3}$ when $b=-1,\frac{1}{2}$ belong to the sub-branch  $I_7\ne0$. So, normalizing the essential torsion coefficients $T^{2}_{13}=1$ by the relation
\begin{equation}\label{2.34}
\begin{array}{ll}
a_{3}=I_7,
\end{array}
\end{equation}
gives rise to the following invariant coframe on the space $M$
\begin{equation}\label{2.35}
\begin{pmatrix}
\theta^{1}\\
\theta^{2}\\
\theta^{3}\\
\theta^4\\
\end{pmatrix}=
\left(\begin{array}{cccc}
\frac{I_7^3}{I_6} & 0 & 0&0 \\
\frac{I_5I_7}{I_6} &I_7& 0 &0 \\
\frac{I_2I_6^2+I_5^2}{2I_6I_7}&\frac{I_1I_6+3I_5}{3I_7} &\frac{I_6}{I_7}&0\\
\frac{I_3I_7^2}{I_6}&0&0&\frac{I_7^2}{I_6}
\end{array}\right)
\begin{pmatrix}
\omega^{1}\\
\omega^{2}\\
\omega^{3}\\
\omega^{4}
\end{pmatrix}.
\end{equation}
The invariant coframe corresponding to the following three canonical forms in this branch satisfies structure equations with constant structure functions.
\subsubsection{The canonical form $u'''=3\frac{u''^2}{u'}+\alpha\frac{(2xu''+u')^{3/2}}{x^2\sqrt{u'}},~\alpha\ne0$.}
\begin{equation}\label{2.36}
\begin{array}{lll}
&d\theta^{1}=\frac{1}{2}~\theta^{1}\wedge\theta^{2}+\frac{1}{4}~\theta^1\wedge\theta^4-\theta^2\wedge\theta^4,\\
&d\theta^{2}=\lambda~\theta^{1}\wedge\theta^{2}+\theta^{1}\wedge\theta^{3}-\frac{\lambda}{2}~\theta^{1}\wedge\theta^{4}-2~\theta^2\wedge\theta^3-\theta^{3}\wedge\theta^{4},\\
&d\theta^{3}=\frac{1}{2}~\theta^{2}\wedge\theta^{3}-\frac{\lambda}{2}~\theta^2\wedge\theta^{4}-\frac{1}{4}~\theta^{3}\wedge\theta^{4},\\
&d\theta^{4}=\theta^1\wedge\theta^3-\lambda~\theta^1\wedge\theta^4-2~\theta^3\wedge\theta^4,
\end{array}
\end{equation}
where $\lambda=\frac{3\alpha^2+4}{16\alpha^2}$.\\
\\
We remark here that  $T^2_{12}=\lambda$ is an invariant relation, which enables the determination of the parameter $\alpha$ appearing in the canonical form.
\subsubsection{The canonical form $u'''=u''^{\frac{3}{2}}.$}
\begin{equation}\label{2.37}
\begin{array}{lll}
&d\theta^{1}=\frac{1}{2}~\theta^{1}\wedge\theta^{2}+\frac{1}{4}~\theta^1\wedge\theta^4-\theta^2\wedge\theta^4,\\
&d\theta^{2}=\frac{3}{16}~\theta^{1}\wedge\theta^{2}+\theta^{1}\wedge\theta^{3}-\frac{3}{32}~\theta^{1}\wedge\theta^{4}-2~\theta^2\wedge\theta^3-\theta^{3}\wedge\theta^{4},\\
&d\theta^{3}=\frac{1}{2}~\theta^{2}\wedge\theta^{3}-\frac{3}{32}~\theta^2\wedge\theta^{4}-\frac{1}{4}~\theta^{3}\wedge\theta^{4},\\
&d\theta^{4}=\theta^1\wedge\theta^3-\frac{3}{16}~\theta^1\wedge\theta^4-2~\theta^3\wedge\theta^4.
\end{array}
\end{equation}
\subsubsection{The canonical form $u'''= u''^3$.}
\begin{equation}\label{2.38}
\begin{array}{lll}
&d\theta^{1}=-\frac{2}{5}~\theta^{1}\wedge\theta^{2}-15~\theta^1\wedge\theta^3+\frac{1}{25}~\theta^1\wedge\theta^4-\theta^2\wedge\theta^4,\\
&d\theta^{2}=\frac{2}{125}~\theta^{1}\wedge\theta^{2}+\theta^{1}\wedge\theta^{3}-5~\theta^{2}\wedge\theta^{3}+\frac{2}{25}~\theta^2\wedge\theta^4-\theta^{3}\wedge\theta^{4},\\
&d\theta^{3}=-\frac{4}{3125}~\theta^{1}\wedge\theta^{2}-\frac{7}{125}~\theta^1\wedge\theta^{3}+\frac{1}{5}~\theta^{2}\wedge\theta^{3}+\frac{3}{25}~\theta^3\wedge\theta^4,\\
&d\theta^{4}=\theta^1\wedge\theta^3-\frac{1}{125}~\theta^1\wedge\theta^4+\frac{3}{5}~\theta^2\wedge\theta^4+10~\theta^3\wedge\theta^4.
\end{array}
\end{equation}
This proves Theorem  \ref{Th3.3} in section 3, see \cite{Olver1995} [Theorem 8.15, page 268].
\subsection{Branch $I_{4}= 0$, $I_6=0$ and $I_5\ne0$:}
 We can eliminate the group parameter $a_{3}$ by normalizing \\
$T^4_{12}=1$  by  setting
\begin{equation}\label{2.39}
\begin{array}{lll}
a_{3}=J_5,
\end{array}
\end{equation}
where $I_5=J_5^2=-I_3^2-I_{3_p}$.
Hence, the structure group is reduced to
\begin{equation}\label{2.40}
{G}_4=\left\{
\left(\begin{array}{cccc}
a_1 & 0 & 0&0 \\
a_2 & J_5 & 0 &0\\
\frac{a_2^2}{2a_1}+\frac{J_5^2}{2a_1}I_2&\frac{a_2J_5}{a_1}+\frac{J_5^2}{3a_1}I_1 & \frac{J_5^2} {a_1}&0\\
\frac{a_1}{J_5}I_3&0&0&\frac{a_1}{J_5}
\end{array}\right)
 \Bigg \vert a_1\neq 0
\right\},
\end{equation}
After performing another loop of reduction and absorption, we deduce the structure equations 
\begin{equation}\label{2.41}
\begin{array}{ll}
d\left(\begin{array}{c}
\theta^1 \\
\theta^2 \\
\theta^3\\
\theta^4
\end{array}\right)=\left(\begin{array}{cccc}
\pi^{\prime1} & 0 & 0&0 \\
\pi^{\prime2} & 0& 0&0 \\
0&\pi^{\prime2} &-\pi^{\prime1}&0\\
0&0&0&\pi^{\prime1}
\end{array}\right) \wedge\left(\begin{array}{c}
\theta^1 \\
\theta^2 \\
\theta^3\\
\theta^4
\end{array}\right)
+\left(\begin{array}{c}
-\theta^2 \wedge \theta^4 \\
T^2_{24}~\theta^2 \wedge \theta^4-\theta^3\wedge\theta^4 \\
T^3_{34}~\theta^3 \wedge \theta^4\\
\theta^1 \wedge \theta^2+T^4_{14}~\theta^1 \wedge \theta^4
\end{array}\right),
\end{array}
\end{equation}
and the essential torsion coefficient
\begin{equation}\label{2.42}
\begin{array}{ll}
T^2_{24}=\frac{I_7}{a_1},\\
\end{array}
\end{equation}
where
\begin{equation}\label{2.43}
\begin{array}{lll}
I_7=\frac{1}{3}I_1J_5-J_{5_x}-pJ_{5_u}+2qI_3J_5.
\end{array}
\end{equation}
It is noted that $I_7$ is a relative invariant and $L_{6;1}$ belong to the sub-branch $I_{7}= 0$.\\
\\
In this sub-branch, no unabsorbable torsion remains, and hence the remaining group variables $a_1$ and $a_2$ cannot be normalized. Moreover, the forms $\pi^{\prime 1}$ and $\pi^{\prime 2}$ are uniquely determined, so that the problem becomes determinate. Consequently, this yields the following $e$-structure on the six-dimensional prolonged space
$M^ {(1)} = M\times G_4$
\begin{equation}\label{2.44}
\left(\begin{array}{l}
\theta^1 \\
\theta^2 \\
\theta^3 \\
\theta^4\\
\pi^{\prime1}\\
\pi^{\prime 2}
\end{array}\right)=\left(\begin{array}{cccccc}
a_1 & 0 & 0&0&0&0 \\
a_2 & J_5 & 0 &0&0&0\\
\frac{a_2^2}{2a_1}+\frac{J_5^2}{2a_1}I_2& \frac{a_2J_5}{a_1}+\frac{J_5^2}{3a_1}I_1 & \frac{J_5^2} {a_1}&0&0&0\\
\frac{a_1}{J_5}I_3&0&0&\frac{a_1}{J_5}&0&0\\
I_8&I_3&0&-\frac{a_2}{J_5}&\frac{1}{a_1}&0\\
I_9&I_{10}&\frac{I_3J_5}{a_1}&-\frac{1}{2a_1}(I_2J_5+\frac{a_2^2}{J_5})&0&\frac{1}{a_1}
\end{array}\right)\left(\begin{array}{c}
\omega^1 \\
\omega^2 \\
\omega^3\\
\omega^4\\
da_1\\
da_2
\end{array}\right),
\end{equation}
where
\begin{equation}\label{2.45}
\begin{array}{lll}
I_8&=\frac{1}{6J_5}(4I_{1_p}J_5-9I_{2_q}J_5+8I_1I_3J_5-6I_3a_2),\\
I_9&=\frac{1}{18J_5a_1}(-3f_{pp}J_5^2-2J_5I_{1_p}(I_1J_5-6a_2)+3J_5I_{2_q}(I_1J_5-9a_2)\\
&-9I_{2_p}J_5^2-2I_3(J_5^2(I_1^2-\frac{9}{2}I_2)-12I_1J_5a_2+\frac{9}{2}a_2^2)),\\
I_{10}&=\frac{1}{3a_1}(J_5(I_{1_p}-3I_{2_q})+3I_3(I_1J_5+a_2)).
\end{array}
\end{equation}
This gives rise to the structure equations
\begin{equation}\label{2.46}
\begin{array}{lll}
&d\theta^{1}=-\theta^{1}\wedge\pi^{\prime1}-\theta^2\wedge\theta^4,\\
&d\theta^{2}=-\theta^{1}\wedge\pi^{\prime2}-\theta^{3}\wedge\theta^{4},\\
&d\theta^{3}=-\theta^{2}\wedge\pi^{\prime2}+\theta^3\wedge\pi^{\prime1},\\
&d\theta^{4}=\theta^1\wedge\theta^2-\theta^4\wedge\pi^{\prime1},\\
&d\pi^{\prime1}=-\theta^1\wedge\theta^3+\theta^4\wedge\pi^{\prime2},\\
&d\pi^{\prime2}=-\theta^2\wedge\theta^3-\pi^{\prime1}\wedge\pi^{\prime2}.
\end{array}
\end{equation}
This proves Theorem  \ref{Th3.4} in section 3, see \cite[Theorem 8.15, page 268]{Olver1995}.
\section{Main theorems}
\setcounter{equation}{0}
This section presents the major theorems that were established before 
section.
\begin{theorem} \label{Th3.1}
A third-order ODE $u^{\prime\prime\prime}=f(x,u,p,q)$  is equivalent 
to one of the  canonical forms (i) $u'''=e^{-u''},$ (ii) $u'''= u''^{\frac{b-2}{b-1}}, b\ne-1,0,\frac{1}{2},1,2,$ with four symmetries under point transformation   if and only if  they belong to the branch $I_4\ne 0, I_6\ne0$  and the exterior derivatives of coframe defined on this branch
\begin{equation}\label{3.1}
\begin{pmatrix}
\theta^{1}\\
\theta^{2}\\
\theta^{3}\\
\theta^4\\
\end{pmatrix}=
\left(\begin{array}{cccc}
J_{4}^3J_{6} & 0 & 0&0 \\
\frac{J_4I_5}{J_6} & J_4J_6 & 0 &0 \\
\frac{I_2J_6^4+I_5^2}{2J_4J_6^3}&\frac{I_1J_6^2+3I_5}{3J_4J_6} &\frac{J_6}{J_4}&0\\
I_3J_4^2&0&0&J_4^2
\end{array}\right)
\begin{pmatrix}
\omega^{1}\\
\omega^{2}\\
\omega^{3}\\
\omega^{4}
\end{pmatrix}
\end{equation}
have identical constant structure equations  for appropriate choices of $J_4$ and $J_6$,
where 
\begin{equation}\label{3.2}
\begin{array}{ll}
I_1=-f_q,&
I_2=-\frac{2}{9}I_1^2-f_p-\frac{1}{3}\hat{D}_xI_1,\\
I_3=-\frac{1}{6}I_{1_q},&
I_4=J_4^6=-\frac{1}{3}I_1I_2-f_u-\frac{1}{2}\hat{D}_xI_2,\\
I_5=\frac{1}{3}I_1I_{3_q}-I_3^2-I_{3_p},&
I_6=J_6^2=-I_{3_q}.
\end{array}
\end{equation}
Moreover, the parameter $b$ appearing in the canonical form (ii) can be evaluated by the invariant relation $T^1_{12}T^2_{23}=-\frac{1}{b}.$
\end{theorem}
\begin{theorem} \label{Th3.2}
A third-order ODE $u^{\prime\prime\prime}=f(x,u,p,q)$  is equivalent 
to one of the canonical forms (i) $u'''=\alpha\frac{u''^2}{u'},\alpha\ne0,\frac{3}{2},3,$  (ii) $u'''=u''^2,$ with four symmetries under point transformation   if and only if  they belong to the branch $ I_4\ne0, I_5\ne0, I_6= 0, J_{4_q}\ne0$  and the exterior derivatives of coframe defined on this branch
\begin{equation}\label{3.3}
\begin{pmatrix}
\theta^{1}\\
\theta^{2}\\
\theta^{3}\\
\theta^4\\
\end{pmatrix}=
\left(\begin{array}{cccc}
J_4J_5 & 0 & 0&0 \\
\frac{J_4J_5^2I_7}{J_{4_q}} &J_5& 0 &0 \\
\frac{J_5(J_4^2J_5^2I_7^2+I_2J_{4_q}^2)}{2J_{4_q}^2J_4}&\frac{J_5(3J_4J_5I_7+I_1J_{4_q})}{3J_{4_q}J_4} &\frac{J_5}{J_4}&0\\
I_3J_4&0&0&J_4
\end{array}\right)
\begin{pmatrix}
\omega^{1}\\
\omega^{2}\\
\omega^{3}\\
\omega^{4}
\end{pmatrix}
\end{equation}
have identical constant structure equations  for appropriate choices of $J_4$ and $J_5$,
where  
\begin{equation}\label{3.4}
\begin{array}{ll}
I_1=-f_q,&
I_2=-\frac{2}{9}I_1^2-f_p-\frac{1}{3}\hat{D}_xI_1,\\
I_3=-\frac{1}{6}I_{1_q},&
I_4=J_4^3=-\frac{1}{3}I_1I_2-f_u-\frac{1}{2}\hat{D}_xI_2,\\
I_5=-J_5^2=-I_3^2-I_{3_p},&
I_6=-I_{3_q},\\
I_7=-\frac{1}{3J_4J_5^2}(3I_3J_4J_5+I_1J_{4_q}J_5-3J_{4_p}J_5).
\end{array}
\end{equation}
Furthermore, the parameter $\alpha$ appearing in the canonical form (i) can be evaluated by the invariant relation $\frac{J_{4_q}}{J_5}=\frac{\beta}{\gamma}$, where  $\beta^3=2\alpha^3-9\alpha^2+9\alpha,~\gamma^2=\alpha^2-3\alpha$.
\end{theorem}
\begin{theorem} \label{Th3.3}
A third-order ODE $u^{\prime\prime\prime}=f(x,u,p,q)$  is equivalent 
to one of  the canonical forms (i) $u'''=3\frac{u''}{u'}+\alpha\frac{(2xu''+u')^{3/2}}{x^2\sqrt{u'}},\alpha\ne0,$ (ii) $u'''= u''^{\frac{3}{2}},$
 (iii) $u'''= u''^3,$ with four symmetries under point transformation   if and only if they belong to the branch $I_4= 0, I_6\ne0$, $I_7\ne0$, and the exterior derivatives of coframe defined on this branch
\begin{equation}\label{3.5}
\begin{pmatrix}
\theta^{1}\\
\theta^{2}\\
\theta^{3}\\
\theta^4\\
\end{pmatrix}=
\left(\begin{array}{cccc}
\frac{I_7^3}{I_6} & 0 & 0&0 \\
\frac{I_5I_7}{I_6} &I_7& 0 &0 \\
\frac{I_2I_6^2+I_5^2}{2I_6I_7}&\frac{I_1I_6+3I_5}{3I_7} &\frac{I_6}{I_7}&0\\
\frac{I_3I_7^2}{I_6}&0&0&\frac{I_7^2}{I_6}
\end{array}\right)
\begin{pmatrix}
\omega^{1}\\
\omega^{2}\\
\omega^{3}\\
\omega^{4}
\end{pmatrix}
\end{equation}
have identical constant structure equations, where
\begin{equation}\label{3.6}
\begin{array}{ll}
I_1=-f_q,&
I_2=-\frac{2}{9}I_1^2-f_p-\frac{1}{3}\hat{D}_xI_1,\\
I_3=-\frac{1}{6}I_{1_q},&
I_4=-\frac{1}{3}I_1I_2-f_u-\frac{1}{2}\hat{D}_xI_2,\\
I_5=\frac{1}{3}I_1I_{3_q}-I_3^2-I_{3_p},&
I_6=-I_{3_q},\\
I_7=-I_3-\left(\frac{I_5}{I_6}\right)_q.
\end{array}
\end{equation}
Moreover, the parameter $\alpha$ appearing in the canonical form (i) can be evaluated by the invariant relation $T^2_{12}=\lambda$, where $\lambda=\frac{3\alpha^2+4}{16\alpha^2}$.
\end{theorem}
\begin{theorem} \label{Th3.4}
A third-order ODE $u^{\prime\prime\prime}=f(x,u,p,q)$  is equivalent 
to the canonical form $u'''=\frac{3}{2}\frac{u''^2}{u'},$ with six symmetries under point transformation  if and only if  it belongs to the branch $I_4=0$, $I_5\ne0$, $I_6=0$, and $I_7=0,$
where
\begin{equation}\label{3.7}
\begin{array}{ll}
I_1=-f_q,&
I_2=-\frac{2}{9}I_1^2-f_p-\frac{1}{3}\hat{D}_xI_1,\\
I_3=-\frac{1}{6}I_{1_q},&
I_4=-\frac{1}{3}I_1I_2-f_u-\frac{1}{2}\hat{D}_xI_2,\\
I_5=-J_5^2=-I_3^2-I_{3_p},&
I_6=-I_{3_q},\\
I_7=\frac{1}{3}I_1J_5-J_{5_x}-pJ_{5_u}+2qI_3J_5.&\\
\end{array}
\end{equation}
\end{theorem}
\section{Construction of Point transformations induced by invariant coframes}
\setcounter{equation}{0}
This section outlines a method for establishing the point transformation between two comparable third-order ordinary differential equations that correspond to one of the canonical forms shown in Table 1. This methodology relies on the invariant coframes presented in the preceding section and the subsequent proposition.
\begin{proposition} 
Assume that the third-order ODEs
\begin{equation}\label{4.1}
\begin{array}{l}
u^{\prime \prime\prime}=f\left(x, u, u^{\prime},u^{\prime \prime}\right), \quad \bar{u}^{\prime \prime\prime}=\bar{f}\left(\bar{x}, \bar{u}, \bar{u}^{\prime},\bar{u}^{\prime \prime}\right),
\end{array}
\end{equation}
are equivalent with respect to a point transformation
\begin{equation}\label{4.2}
\bar x = \varphi (x,u),\,\,\bar u = \psi (x,u),\quad \varphi_{x}\psi_{u}-\varphi_{u}\psi_{x}\ne 0.
\end{equation}
Given an invariant four-dimensional coframe on the space $M$ such that
\begin{equation}\label{4.3}
\Phi^*
\left( {\begin{array}{*{20}c}
   {\bar a_1 } & 0 & 0 &0 \\
   {\bar a_2 } & {\bar a_3 } & 0&0  \\
   {\bar a_4 } &  {\bar a_5 } & {\bar a_6 }&0  \\
 {\bar a_7 }&0&0& {\bar a_9 }
\end{array}} \right)\left( {\begin{array}{*{20}c}
   {d\bar u - \bar p\,d\bar x}  \\
   {d\bar p - \bar q\,d\bar x}  \\
 {d\bar q - \bar f\,d\bar x}  \\
   {d\bar x}  \\
\end{array}} \right) = \left( {\begin{array}{*{20}c}
   {a_1 } & 0 & 0 &0 \\
   {a_2 } & {a_3 } & 0 &0 \\
   {a_4 } & {a_5}& {a_6 }&0  \\
{a_7}&0&0&{a_9}
\end{array}} \right)\left( {\begin{array}{*{20}c}
   {du - p\,dx}  \\
   {dp - q\,dx}  \\
 {dq - f\,dx} \\
   {dx}  \\
\end{array}} \right)
\end{equation}
for some functions $a_i(x,u,p,q),~ i=1,2 ,3 ,4,5, 6,7,9$ where $\Phi^*$  denotes the pullback  defined by the second prolongation of the point transformation (\ref{4.2}).
Consequently, the following system determines the point transformation between the third-order ODEs (\ref{4.1}).
\begin{equation}\label{4.4}
\begin{array}{ll}
\hat{D}_x\eta=b_9~\bar{f},\,\, \eta_u=b_4+b_7~\bar{f},\,\, \eta_p=b_5,\,\,\eta_q=b_6,\\
\hat{D}_x\chi=b_9~\eta,\,\, \chi_u=b_2+b_7~\eta,\,\, \chi_p=b_3,\\
\hat{D}_x\varphi=b_9,\,\, \varphi_u=b_7,\,\, \\
{b_7} \hat{D}_x\psi={b_9}(\psi_u-b_1),
\end{array}
\end{equation}
where $$\chi=\frac{\hat{D}_x\psi}{\hat{D}_x\varphi },\,\, \eta=\frac{\hat{D}_x\chi}{\hat{D}_x\varphi },\,\,\bar{f}=\frac{\hat{D}_x\eta}{\hat{D}_x\varphi} ,\,\,\hat{D}_x  = \frac{\partial }{{\partial x}} + p\frac{\partial }{{\partial u}} + q\frac{\partial }{{\partial p}} + f\frac{\partial }{{\partial q}}.
$$
and
\begin{equation}\label{4.5}
\left( {\begin{array}{*{20}c}
   {b_1 } & 0 & 0&0 \\
   {b_2 } & {b_3 } & 0&0  \\
   {b_4 } & {b_5}& {b_6 }&0  \\
{b_7} &0&0&{b_9}
\end{array}} \right) = \left( {\begin{array}{*{20}c}
   {\bar a_1 } & 0 & 0 &0 \\
   {\bar a_2 } & {\bar a_3 } & 0&0  \\
   {\bar a_4 } & {\bar a_5 } & {\bar a_6 } &0 \\
{\bar a_7 }&0&0&{\bar a_9 }
\end{array}} \right)^{ - 1} \left( {\begin{array}{*{20}c}
   {a_1 } & 0 & 0&0  \\
   {a_2 } & {a_3 } & 0&0  \\
   {a_4 } &   {a_5 }& {a_6 }&0  \\
 {a_7}&0&0& {a_9}
\end{array}} \right).
\end{equation}
\proof
Equation (\ref{4.3}) can be represented in terms of (\ref{4.5})  as
\begin{equation}\label{4.6}
\left( {\begin{array}{*{20}c}
   {d\bar u - \bar p\,d\bar x}  \\
   {d\bar p - \bar{q}\,d\bar x}  \\
{d\bar q-\bar f\,d\bar x}\\
   {d\bar x}  \\
\end{array}} \right) = \left( {\begin{array}{*{20}c}
   {b_1 } & 0 & 0&0  \\
   {b_2 } & {b_3 } & 0&0  \\
   {b_4 } & {b_5}& {b_6 }&0  \\
{b_7} &0&0&{b_9}
\end{array}} \right)\left( {\begin{array}{*{20}c}
   {du - p\,dx}  \\
   {dp - q\,dx}  \\
{dq-f\,dx}\\
   {dx}  \\
\end{array}} \right).
\end{equation}
On the other hand, computing the pullback of the left-hand side of equation~\eqref{4.6} yields the system~\eqref{4.4}.
\endproof
\end{proposition} 
\begin{proposition} 
Assume that the third-order ODEs
\begin{equation}\label{4.7}
\begin{array}{l}
u^{\prime \prime\prime}=f\left(x, u, u^{\prime},u^{\prime \prime}\right), \quad \bar{u}^{\prime \prime\prime}=\frac{3}{2}\frac{\bar{q}^2}{\bar{p}},
\end{array}
\end{equation}
are equivalent with respect to a point transformation
\begin{equation}\label{4.8}
\bar x = \varphi (x,u),\,\,\bar u = \psi (x,u),\quad \varphi_{x}\psi_{u}-\varphi_{u}\psi_{x}\ne 0.
\end{equation}
Consequently, the following system determines the point transformation between the third-order ODEs (\ref{4.7})
\begin{equation}\label{4.9}
\begin{array}{ll}
\hat{D}_xa_2=-\frac{b_{16}}{b_{17}}\,\,  a_{2_u}=-\frac{b_{13}}{b_{17}},\,\,  a_{2_p}=-\frac{b_{14}}{b_{17}},\,\, a_{2_q}=-\frac{b_{15}}{b_{17}},
\end{array}
\end{equation}
\begin{equation}\label{4.10}
\begin{array}{ll}
\hat{D}_xa_1=-\frac{b_{11}}{b_{12}},\,\,  a_{1_u}=-\frac{b_9}{b_{12}},\,\,  a_{1_p}=-\frac{b_{10}}{b_{12}}, a_{1_q}=0,\\
\end{array}
\end{equation}
\begin{equation}\label{4.11}
\begin{array}{ll}
\hat{D}_x\eta=b_8~\bar{f},\,\, \eta_u=b_4+b_7~\bar{f},\,\, \eta_p=b_5,\,\,\eta_q=b_6,\\
\hat{D}_x\chi=b_8~\eta,\,\, \chi_u=b_2+b_7~\eta,\,\, \chi_p=b_3,\\
\hat{D}_x\varphi=b_8, \varphi_u=b_7,\\
 {b_7} \hat{D}_x\psi={b_8}(\psi_u-b_1),
\end{array}
\end{equation}
where
 $a_1(x,u,p),a_2(x,u,p,q)$ are auxiliary functions, $\chi=\frac{\hat{D}_x\psi}{\hat{D}_x\varphi }, \eta=\frac{\hat{D}_x\chi}{\hat{D}_x\varphi },$\\$\bar{f}=\frac{3}{2}\frac{\eta^2}{\chi},$
and
\begin{equation}\label{4.12}
\begin{aligned}
&\left( {\begin{array}{*{20}c}
   {b_1 } & 0 & 0&0 &0&0\\
   {b_2 } & {b_3 } & 0&0 &0&0 \\
   {b_4 } & {b_5}& {b_6 }&0&0&0  \\
{b_7} &0&0&{b_8}&0&0\\
b_9&b_{10}&0&b_{11}&b_{12}&0\\
b_{13}&b_{14}&b_{15}&b_{16}&0&b_{17}
\end{array}} \right) 
=
\left(\begin{array}{cccccc}
1 & 0 & 0&0&0&0 \\
0 & \frac{1}{2\chi} & 0 &0&0&0\\
0& -\frac{\eta}{4\chi^3}& \frac{1} {4\chi^2}&0&0&0\\
1&0&0&2\chi&0&0\\
0&\frac{1}{2\chi}&0&0&1&0\\
0& -\frac{\eta}{4\chi^3}&\frac{1}{4\chi^2}&0&0&1\\
\end{array}\right)^{ - 1}\\
 &\left(\begin{array}{cccccc}
a_1 & 0 & 0&0&0&0 \\
a_2 & J_5 & 0 &0&0&0\\
\frac{a_2^2}{2a_1}+\frac{J_5^2}{2a_1}I_2& \frac{a_2J_5}{a_1}+\frac{J_5^2}{3a_1}I_1 & \frac{J_5^2} {a_1}&0&0&0\\
\frac{a_1}{J_5}I_3&0&0&\frac{a_1}{J_5}&0&0\\
I_8&I_3&0&-\frac{a_2}{J_5}&\frac{1}{a_1}&0\\
I_9&I_{10}&\frac{I_3J_5}{a_1}&-\frac{1}{2a_1}(I_2J_5+\frac{a_2^2}{J_5})&0&\frac{1}{a_1}
\end{array}\right),
\end{aligned}
\end{equation}
\begin{equation*}
\begin{array}{lll}
I_1=-f_q,\\
I_2=-\frac{2}{9}I_1^2-f_p-\frac{1}{3}\hat{D}_xI_1,\\
I_3=-\frac{1}{6}I_{1_q},\\
I_5=J_5^2=-I_3^2-I_{3_p},\\
I_8=\frac{1}{6J_5}(4I_{1_p}J_5-9I_{2_q}J_5+8I_1I_3J_5-6I_3a_2),\\
I_9=\frac{1}{18J_5a_1}(-3f_{pp}J_5^2-2J_5I_{1_p}(I_1J_5-6a_2)+3J_5I_{2_q}(I_1J_5-9a_2)\\
~~~~-9I_{2_p}J_5^2-2I_3(J_5^2(I_1^2-\frac{9}{2}I_2)-12I_1J_5a_2+\frac{9}{2}a_2^2)),\\
I_{10}=\frac{1}{3a_1}(J_5(I_{1_p}-3I_{2_q})+3I_3(I_1J_5+a_2)).\\
\end{array}
\end{equation*}
\proof
The equivalence of the third-order  ODEs (\ref{4.7}) under the point transformation (\ref{4.8}) can be checked by using the  invariant six-dimensional coframe (\ref{2.44}) on the space $M^{(1)}=M\times G_4$ such that
\begin{equation}\label{4.13}
\begin{aligned}
&\Phi^*\left(\begin{array}{cccccc}
\bar{a}_1 & 0 & 0&0&0&0 \\
\bar{a}_2 &\bar{J}_5 & 0 &0&0&0\\
\frac{\bar{a}_2^2}{2\bar{a}_1}+\frac{\bar{J}_5^2}{2\bar{a}_1}\bar{I}_2& \frac{\bar{a}_2\bar{J}_5}{\bar{a}_1}+\frac{\bar{J}_5^2}{3\bar{a}_1}\bar{I}_1 & \frac{\bar{J}_5^2} {\bar{a}_1}&0&0&0\\
\frac{\bar{a}_1}{\bar{J}_5}\bar{I}_3&0&0&\frac{\bar{a}_1}{\bar{J}_5}&0&0\\
\bar{I_8}&\bar{I}_3&0&-\frac{\bar{a}_2}{\bar{J}_5}&\frac{1}{\bar{a}_1}&0\\
\bar{I_9}&\bar{I_{10}}&\frac{\bar{I}_3\bar{J}_5}{\bar{a}_1}&-\frac{1}{2\bar{a}_1}(\bar{I}_2\bar{J}_5+\frac{\bar{a}_2^2}{\bar{J}_5})&0&\frac{1}{\bar{a}_1}
\end{array}\right)\left(\begin{array}{c}
 {d\bar u - \bar p\,d\bar x}  \\
   {d\bar p - \bar{q}\,d\bar x}  \\
{d\bar q-\bar f\,d\bar x}\\
   {d\bar x}  \\
d\bar{a}_1\\
d\bar{a}_2
\end{array}\right)=\\
&\left(\begin{array}{cccccc}
a_1 & 0 & 0&0&0&0 \\
a_2 & J_5 & 0 &0&0&0\\
\frac{a_2^2}{2a_1}+\frac{J_5^2}{2a_1}I_2& \frac{a_2J_5}{a_1}+\frac{J_5^2}{3a_1}I_1 & \frac{J_5^2} {a_1}&0&0&0\\
\frac{a_1}{J_5}I_3&0&0&\frac{a_1}{J_5}&0&0\\
I_8&I_3&0&-\frac{a_2}{J_5}&\frac{1}{a_1}&0\\
I_9&I_{10}&\frac{I_3J_5}{a_1}&-\frac{1}{2a_1}(I_2J_5+\frac{a_2^2}{J_5})&0&\frac{1}{a_1}
\end{array}\right)\left(\begin{array}{c}
 {du - p\,dx}  \\
   {dp - q\,dx}  \\
{dq-f\,dx}\\
   {dx}  \\
da_1\\
da_2
\end{array}\right),
\end{aligned}
\end{equation}
 where $\Phi^*$  denotes the pullback defined by the second prolongation of the point transformation (\ref{4.8}). 
 By incorporating the values $\bar{I_1}=-\frac{3\eta}{\chi},~\bar{I_2}=0,~\bar{I_3}=\frac{1}{2\chi},~\bar{J_5}=\frac{1}{2\chi},~\bar{I_8}=0,~\bar{I_9}=0,~\bar{I}_{10}=-\frac{\eta}{4\chi^3}$ for $\bar{u}^{\prime \prime\prime}=\frac{3}{2}\frac{\bar{q}^2}{\bar{p}}$ and  $\bar{a_1}=1,~\bar{a_2}=0$,  equation (\ref{4.13}) can be expressed as
\begin{equation}\label{4.14}
\left( {\begin{array}{*{20}c}
   {d\bar u - \bar p\,d\bar x}  \\
   {d\bar p - \bar{q}\,d\bar x}  \\
{d\bar q-\bar f\,d\bar x}\\
   {d\bar x}  \\
0\\
0
\end{array}} \right) = \left( {\begin{array}{*{20}c}
   {b_1 } & 0 & 0&0 &0&0\\
   {b_2 } & {b_3 } & 0&0 &0&0 \\
   {b_4 } & {b_5}& {b_6 }&0&0&0  \\
{b_7} &0&0&{b_8}&0&0\\
b_9&b_{10}&0&b_{11}&b_{12}&0\\
b_{13}&b_{14}&b_{15}&b_{16}&0&b_{17}
\end{array}} \right)\left( {\begin{array}{*{20}c}
   {du - p\,dx}  \\
   {dp - q\,dx}  \\
{dq-f\,dx}\\
   {dx}  \\
da_1\\
da_2
\end{array}} \right).
\end{equation}
On the other hand,  computing the pullback of the left-hand side of equation (\ref{4.14}) yields systems (\ref{4.9}),  (\ref{4.10}),  (\ref{4.11}).
\endproof
\end{proposition} 
\begin{remark}
Utilising systems (\ref{4.9}), (\ref{4.10}) and the equations  
\begin{equation}
{b_7} \hat{D}_x\psi={b_8}(\psi_u-b_1), \chi_p=b_3
\end{equation} 
given in system (\ref{4.11}),  we obtain 
\begin{equation}
a_1(x,u,p)=\frac{\varphi_x\psi_u-\varphi_u\psi_x}{\hat{D}_x\varphi}, J_5(x,u,p)=\frac{\chi_p}{2\chi}, a_{1_p}=\frac{a_1}{J_5}a_{2_q}, a_{2_{qq}}=0, a_{1_p}=a_1(J_5-I_3).
\end{equation} 
Therefore
\begin{equation}\label{4.17}
a_2(x,u,p,q)=J_5(J_5-I_3) q+A(x,u,p),
\end{equation}
for some function $A(x,u,p).$
\end{remark}

We now present examples to illustrate our main results.

\begin{example}
Consider the class of non-linear third-order ODE
\begin{equation}\label{4.18}
	u'''=\frac{1}{u'}(3u''^2- u'^5e^{\frac{u''}{u'^3}}).
\end{equation}
The function $f(x,u,p,q)=\frac{1}{p}(3q^2- p^5e^{\frac{q}{p^3}})$ satisfies the conditions of  Theorem 3.1 for appropriate choices of $J_4$ and $J_6$. It is equivalent to the canonical form 
\begin{equation}\label{4.19}
\bar{u}'''=e^{-\bar{u}''},
\end{equation}
 with four symmetries under point transformation.\\
 \\
To construct the point transformation (\ref{4.2}), we must calculate the entries of the matrix in the left hand side of (\ref{4.5}) by invoking equations (\ref{4.18}) and (\ref{4.19}) as outlined below.
\begin{equation}\label{4.20}
\left( {\begin{array}{*{20}c}
   {b_1 } & 0 & 0&0  \\
   {b_2 } & {b_3 } & 0&0  \\
   {b_4 } & {b_5}& {b_6 }&0  \\
{b_7} &0&0&{b_9}
\end{array}} \right) =\begin{pmatrix}
  -\frac{1}{p}e^{\frac{2(\eta p^3+q)}{p^3}} & 0 & 0 &0 \\
  \frac{q}{p^3}e^{\frac{\eta p^3+q}{p^3}}& -\frac{1}{p^2}e^{\frac{\eta p^3+q}{p^3}}& 0&0  \\
  - e^{\frac{q}{p^3}}& \frac{3q}{p^4}& -\frac{1}{p^3}&0  \\
e^{\frac{\eta p^3+q}{p^3}}&0&0&pe^{\frac{\eta p^3+q}{p^3}}
\end{pmatrix}
\end{equation}
Solving  system (\ref{4.4}) yields the point transformation 
.\begin{equation}\label{4.21}
\bar x =u,~\bar u =x.
\end{equation}
\end{example}
\begin{example}
Consider the class of non-linear third-order ODE
\begin{equation}\label{4.22}
	u'''=\frac{1}{u^2u'}(u^2u''^2+2uu'^2u''-2u'^4).
\end{equation}
The function $f(x,u,p,q)=\frac{1}{u^2p}(u^2q^2+2up^2q-2p^4)$ satisfies the conditions of Theorem 3.2 for appropriate choices of $J_4$ and $J_5$. Moreover, since $\frac{\beta}{\gamma}=-\frac{i}{\sqrt[6]{2}}$,  it is equivalent to the canonical form 
\begin{equation}\label{4.23}
\bar{u}'''=\frac{\bar{u}''^2}{\bar{u}'},
\end{equation}
 with four symmetries under point transformation.\\
 \\
To construct the point transformation (\ref{4.2}), we must calculate the entries of the matrix in the left hand side of (\ref{4.5}) by using equations (\ref{4.22}) and (\ref{4.23}) as outlined below.
\begin{equation}\label{4.24}
\left( {\begin{array}{*{20}c}
   {b_1 } & 0 & 0&0  \\
   {b_2 } & {b_3 } & 0&0  \\
   {b_4 } & {b_5}& {b_6 }&0  \\
{b_7} &0&0&{b_9}
\end{array}} \right) =\begin{pmatrix}
  -\frac{\chi^2(2p^2-qu)}{\eta up^2}& 0 & 0 &0 \\
 -\frac{2\chi}{u}& \frac{\chi}{p}& 0&0  \\
 -\frac{2\eta(3p^2-qu)}{u(2p^2-qu)}&\frac{4\eta p}{2p^2-qu}&-\frac{\eta u}{2p^2-qu}&0  \\
0&0&0&-\frac{\chi(2p^2-qu)}{\eta up}
\end{pmatrix}
\end{equation}
Solving  system (\ref{4.4}) gives the point transformation 
.\begin{equation}\label{4.25}
\bar x =-x,~\bar u =\frac{1}{u}.
\end{equation}
\end{example}
\begin{example}
Consider the class of non-linear third-order ODE
\begin{equation}\label{4.26}
	u'''=\frac{3u'u''^2}{1+u'^2}.
\end{equation}
The function $f(x,u,p,q)=\frac{3pq^2}{1+p^2}$ satisfies the conditions of Theorem 3.4. So it is equivalent to the canonical form 
\begin{equation}\label{4.27}
\bar{u}'''=\frac{3}{2} \frac{\bar{u}''^2}{\bar{u}'}
\end{equation}
 with six point symmetries under point transformation.\\
 \\
To find the point transformation (\ref{4.8}), we must calculate the entries of matrix in the left hand side of (\ref{4.12}) by using equations (\ref{4.26}) as outlined below.
\begin{equation}
\begin{aligned}
 &\left( {\begin{array}{*{20}c}
   {b_1 } & 0 & 0&0 &0&0\\
   {b_2 } & {b_3 } & 0&0 &0&0 \\
   {b_4 } & {b_5}& {b_6 }&0&0&0  \\
{b_7} &0&0&{b_8}&0&0\\
b_9&b_{10}&0&b_{11}&b_{12}&0\\
b_{13}&b_{14}&b_{15}&b_{16}&0&b_{17}
\end{array}} \right)=\\
&\left(\begin{array}{cccccc}
a_1 & 0 & 0&0&0&0 \\
2a_2\chi &2 J_5 \chi& 0 &0&0&0\\
\frac{2I_2J_5^2\chi^2+2a_2^2\chi^2+2a_1a_2\eta}{a_1}& \frac{4J_5(I_1J_5\chi^2+3a_2\chi^2+\frac{3}{2}a_1\eta) }{3a_1}& \frac{4\chi^2J_5^2} {a_1}&0&0&0\\
\frac{a_1(I_3-J_5)}{2J_5\chi }&0&0&\frac{a_1}{2J_5\chi}&0&0\\
I_8-a_2&I_3-J_5&0&-\frac{a_2}{J_5}&\frac{1}{a_1}&0\\
\frac{-I_2J_5^2+2I_9a_1-a_2^2}{2a_1}&\frac{-I_1J_5^2+3I_{10}a_1-3J_5a_2}{3a_1}&\frac{J_5(I_3-J_5)}{a_1}&\frac{-I_2J_5^2-a_2^2}{2J_5a_1}&0&\frac{1}{a_1}
\end{array}\right),
\end{aligned}
\end{equation}
where $I_1=-\frac{6pq}{p^2+1},~I_2=-\frac{q^2}{p^2+1},~I_3=\frac{p}{p^2+1},~J_5=\frac{i}{p^2+1}$,  for $u'''=\frac{3u'u''^2}{1+u'^2}$.
Utilising Remark 4.3, $a_2=-\frac{iq}{(p+i)(p^2+1)}+A(x,u,p),$ for some function $A(x,u,p)$. Therefore, substituting this value of $a_2$ into the system (\ref{4.9}) gives \\
$$a_2=-\frac{iq}{(p+i)(p^2+1)}.$$
Secondly, solving system (\ref{4.10}) gives $$a_1=-\frac{2}{i+p}.$$
Finally, solving system  (\ref{4.11})  results in the point transformation 
\begin{equation}\label{4.29}
\bar x=u+ix, \bar u=x+iu.
\end{equation}
\end{example}

\section{Conclusion}
In this work, we applied Cartan’s equivalence method to explicitly construct invariant coframes for four distinct branches associated with scalar third-order ordinary differential equations. These invariant coframes provide a complete characterization of all non-linearizable third-order ODEs that admit a four-dimensional Lie symmetry subalgebra under point transformations. 

Building on this geometric framework, we also developed a systematic procedure for constructing the corresponding point transformations directly from the derived invariant coframes. In cases involving prolongation, where not all parameters can be normalized, we showed that the remaining parameters naturally give rise to auxiliary functions, which can then be used to efficiently determine the required point transformation.

The presented examples demonstrate the effectiveness of the method and illustrate how the invariant coframes lead to a deeper understanding of the structure and classification of third-order ODEs that admit a four-dimensional Lie subalgebra.
\subsection*{Acknowledgements}
The authors would like to thank Birzeit University for its support and excellent research facilities. FM thanks Wits for its support.\\\\
Conflict of Interest: The authors declare that they have no conflict of interest.

\section*{Appendix}
\begin{table}[H]
\centering
\begin{tabular}{|l|l|l|}
 \hline
Algebra &Lie point symmetry realizations in plane &Representative equations\\ \hline\hline
$L_{4;1}$&$X_1=\partial_u, X_2=u\partial_u, X_3=f_1(x)\partial_u, X_4=f_2(x)\partial_u$&$u'''=a^3(x)u$, \\
&$f_i'''+a(x)f_i'=0,\,i=1,2$&$a$ not const,~$ \left(\frac{2aa''-3a'^2}{a^4}\right)_x\neq 0$\\
$L_{4;2}$&$X_1=\partial_x, X_2=\partial_u, X_3=u\partial_u, X_4=x\partial_x$&$u'''=\alpha\frac{q^2}{p}, \alpha\ne0,\tfrac{3}{2}$\\
$L_{4;3}$&$X_1=\partial_u, X_2=\partial_x, X_3=x\partial_u,$&$u'''=\alpha q^{b-2\over b-1}, b\ne1,2, \alpha\ne0$\\
& $X_4=x\partial_x+(1+b)u\partial_u, b\ne1$&\\
$L_{4;4}$&$X_1=\partial_u, X_2=\partial_x, X_3=x\partial_x+u\partial_u,$&$u'''={3pq^2\over 1+p^2}+\alpha{q^2\over 1+p^2}, \alpha\ne0$\\
& $X_4=u\partial_x-x\partial_u$&\\
$L_{4;5}^{II}$&$X_1=\partial_u, X_2=x\partial_x+u\partial_u, X_3=x\partial_x,$&$u'''=3 {q^{2}\over p}+\alpha{(2xq+p)^{3/2}\over x^2\sqrt{p}},\alpha\ne0$\\
& $X_4=2xu\partial_x+u^2\partial_u$&\\
$L_{4;6}$&$X_1=\partial_u, X_2=x\partial_u, X_3=\partial_x,$&$u'''=\alpha\exp(-q), \alpha\ne0$\\
& $X_4=u\partial_x+(2u+\frac12x^2)\partial_u$&\\ \hline
$L_{5;1}$&$X_1=\partial_x, X_2=u\partial_u, X_3=f_1(x)\partial_u, X_4=f_2(x)\partial_u,$&$u'''+a_1p+a_0u=0$, \\
& $X_5=f_3(x)\partial_u$,$f_i'''+a_1f_i'+a_0f_i=0$,$i=1,\ldots,3$& $a_i$ consts., $a_0\ne0$\\ \hline
$L_{6;1}$&$X_1=\partial_x, X_2=x\partial_x, X_3=x^2\partial_x,$&$u'''=\frac32 {q^{2}\over p}$\\
& $X_4=\partial_u, X_5=u\partial_u, X_6=u^2\partial_u$&\\
$L_{6;2}$&$X_1=\partial_u, X_2=\partial_x, X_3=x\partial_x+u\partial_u, X_4=u\partial_x-x\partial_u$&$u'''={3pq^{2}\over 1+p^2}$\\
& $X_5=(x^2-u^2)\partial_u+2xu\partial_u, X_6=2xu\partial_x+(u^2-x^2)\partial_u$&\\ \hline
$L_{7;1}$&$X_1=\partial_u, X_2=x\partial_u, X_3=x^2\partial_u, X_4=u\partial_u$&$u'''=0$\\
& $X_5=\partial_x, X_6=x\partial_x, X_7=x^2\partial_x+2xu\partial_u$&\\ \hline
\end{tabular}
\caption{Classification of scalar third-order ODEs admitting non-similar real Lie algebras $L_r$ of dimension $r$, where $r=4, 5, 6, 7$. }
\end{table}

\begin{thebibliography}{99}
\bibitem{Lie1888} Lie, S.  Klassifikation und Integration von gew$\ddot{o}$nlichen Differentialgleichungen zwischen $x, y$, die eine Gruppe von Transformationen gestatten, III. Archiv for Matematik og Naturvidenskab. {\em Archiv for Matematik og Naturvidenskab} {\bf 1883}, {\em 8}, 371-427.
\bibitem{Ibragimov2007}Ibragimov, Nail H., and Meleshko, V.  Invariants and invariant description of second-order ODEs with three infinitesimal symmetries. I. {\em  Communications in Nonlinear Science and Numerical Simulation} {\em 12}, no. 8 (2007): 1370-1378.
\bibitem{Ibragimov2008}Ibragimov, Nail H., and Meleshko, V.  Invariants and invariant description of second-order ODEs with three infinitesimal symmetries. II. {\em Communications in Nonlinear Science and Numerical Simulation} {\em 13}, no. 6 (2008): 1015-1020.
\bibitem{Al-Dweik2025}  Al-Dweik Ahmad Y., Marwan  Aloqueili,  Omar A. Abuloha,   Raddad Batoul M.,   Khalil Sondos R., and Mahomed F. M. ``Invariant Characterization of Scalar Second-Order ODEs That Admit Three Point Symmetry Lie Algebra via Cartan’s Equivalence Method.'' Mathematical Methods in the Applied Sciences 0 (2025): 1-10.
\bibitem{Mahomed1988}Mahomed  Fazal M.,  and P. G. L. Leach. "Normal forms for third order equations." In Proceedings of the Workshop on Finite Dimensional Integrable Nonlinear Dynamical Systems, vol. 178. World Scientific: Singapore, 1988.
\bibitem{Gat1992}Gat Omri. "Symmetry algebras of third‐order ordinary differential equations." Journal of Mathematical Physics 33, no. 9 (1992): 2966-2971.
\bibitem{Schmucker1998}Schmucker  Annett, and Günter Czichowski. "Symmetry algebras and normal forms of third order ordinary differential equations." J. Lie Theory 8 (1998): 129-137.
\bibitem{Ibragimov1996}Ibragimov Nail H., and Mahomed Fazal M. Ordinary differential equations. In CRC Handbook of Lie Group Analysis of Differential Equations, Vol. 3, Ibragimov NH(ed). CRC Press: Boca Raton, 1996; 191–216.
\bibitem{Chern1940} Chern, S. S.  The geometry of the differential equation $y'''=F(x,y,y,y'')$, Sci. Rep. Nat. Tsing Hua Univ. 4 (1940), 97-111.
\bibitem{Neut2002}Neut  Sylvain, and  Petitot Michel. "La géométrie de l'équation $y'''= f (x, y, y', y'')$." Comptes rendus. Mathématique 335, no. 6 (2002): 515-518.
\bibitem{Ibragimov2005}Ibragimov Nail H., and  Meleshko Sergey V.. "Linearization of third-order ordinary differential equations by point and contact transformations." Journal of Mathematical Analysis and Applications 308, no. 1 (2005): 266-289.
\bibitem{Dweik2018_1}Al-Dweik  Ahmad Y.,  Mustafa M. T., and  Mahomed Fazal M. "Invariant characterization of scalar third‐order ODEs that admit the maximal point symmetry Lie algebra." Mathematical Methods in the Applied Sciences 41, no. 12 (2018): 4714-4723.
\bibitem{Dweik2019}Al-Dweik  Ahmad Y.,  Mahomed F. M., and  Mustafa Muhammad T. "Invariant characterization of third-order ordinary differential equations $u'''= f (x, u, u', u'')$ with five-dimensional point symmetry group." Communications in Nonlinear Science and Numerical Simulation 67 (2019): 627-636.
\bibitem{Dweik2018_2}Al-Dweik Ahmad Y., Mustafa M. T.,  Mahomed Fazal M., and Rajai S. Alassar. "Linearization of third‐order ordinary differential equations via point transformations." Mathematical Methods in the Applied Sciences 41, no. 16 (2018): 6955-6967.
\bibitem{Olver1995} Olver, P.J. Equivalence, Invariants and Symmetry, Cambridge University Press, Cambridge, 1995.                           
\bibitem{Neut2003} Neut, S., 2003. Implantation et nouvelles applications de la méthode d'équivalence de Cartan. PhD thesis, Univ. Lille I.
\end{thebibliography}
\end{document}